\documentclass[a4paper,11pt,twoside,leqno]{article}
\usepackage{amsmath,amssymb,amsfonts,amsthm,graphicx,fancyhdr}
\usepackage[latin1]{inputenc}
\usepackage[T1]{fontenc}
\usepackage{rotating}
\usepackage[colorlinks=true]{hyperref}
\usepackage{thmtools}

\newtheorem{teo}{Theorem}

\newtheorem{prop}[teo]{Proposition}
\newtheorem{cor}[teo]{Corollary}

\declaretheoremstyle[
  spaceabove=\topsep, spacebelow=\topsep,
  headfont=\bf,  
  notefont=\mdseries, notebraces={(}{)},
  bodyfont=\rmfamily, 
  postheadspace=1em,
  qed=$\Diamond$
]{drem}
\declaretheorem[style=drem, name=Remark, numberlike=teo]{rmk}

\newcommand{\eg}[0]{\emph{e.g.} }
\newcommand{\ie}[0]{\emph{i.e.} }

\newcommand{\ssl}[1]{\underline{#1}}

\newcommand{\eps}[0]{\varepsilon}

\newcommand{\rr}[0]{\ensuremath{\mathbb{R}}}
\newcommand{\zz}[0]{\ensuremath{\mathbb{Z}}}

\newcommand{\del}[0]{\ensuremath{\partial}}

\newcommand{\vide}[0]{\varnothing}

\newcommand{\comp}[0]{\mathsf{c}}

\newcommand{\IS}[0]{\mathrm{I}\mathrm{S}}

\newcommand{\FP}[0]{\mathrm{F}\!\mathrm{P}\!}

\newcommand{\BHD}[0]{\mathsf{B}\!\mathcal{H}\!\mathsf{D}}
\newcommand{\HD}[0]{\mathcal{H}\!\mathsf{D}}
\newcommand{\D}[0]{\mathsf{D}}
\newcommand{\LL}[0]{\mathsf{L}}

\newcommand{\Gm}[0]{\Gamma}

\begin{document}
\begin{center}
\Large Harmonic functions with finite $p$-energy\\ 
on lamplighter graphs are constant.
\end{center}

\vspace*{1cm}

\renewcommand{\thefootnote}{\fnsymbol{footnote}}

\centerline{\large Antoine Gournay\footnote{TU Dresden, Fachrichtung Mathematik, 01062 Dresden, Germany.}\footnote{Supported by the ERC-StG 277728 ``GeomAnGroup''.}} 

\vspace*{1cm}

\centerline{\textsc{Abstract}}

\begin{center}
\parbox{10cm}{{ \small 
\hspace*{.1ex} The aim of this note is to show that lamplighter graphs where the space graph is infinite and at most two-ended and the lamp graph is at most two-ended do not admit harmonic functions with gradients in $\ell^p$ (\ie finite $p$-energy) for any $p\in [1,\infty[$ except constants (and, equivalently, that their reduced $\ell^p$ cohomology is trivial in degree one). This answers a question of Georgakopoulos \cite{Agelos} on functions with finite energy in lamplighter graphs. The proof relies on a theorem of Thomassen \cite{Tho} on spanning lines in squares of graphs. Using similar arguments, it is also shown that many direct products of graphs (including all direct products of Cayley graphs) do not admit non-constant harmonic function with gradient in $\ell^p$.
}}
\end{center}

\section{Introduction}\label{s-intro}

\setcounter{teo}{0}
\renewcommand{\theteo}{\arabic{teo}}
\renewcommand{\theques}{\arabic{teo}}
\renewcommand{\thecor}{\arabic{teo}}

Given two graphs $H=(X,E)$ (henceforth the ``space'' graph) and $L=(Y,F)$ (henceforth the ``lamp'' graph), the lamplighter graph $G:= L \wr H$ is the graph constructed as follows. Fix some root vertex $o \in Y$ and let $\big( \oplus_X Y \big)$ be the set of ``finitely supported'' functions from $X \to Y$ (\ie only finitely many elements of $X$ are not sent to $o \in Y$.). Its vertices are elements of $X \times \big( \oplus_X Y \big)$. Two vertices $(x,f)$ and $(x',f')$ are adjacent if
\begin{itemize}\setlength{\itemsep}{0ex} \renewcommand{\labelitemi}{ $\cdot$ }
\item either $x \sim x'$ in $H$ and $f=f'$,
\item or $x = x'$, $f(y) = f'(y)$ for all $y \neq x$ and $f(x) \sim f'(x)$ in $L$.
\end{itemize}
It is easy to see that $L \wr H$ is connected exactly when both $H$ and $L$ are. In fact, in this note, all graphs will be assumed to be connected (this is not important) and the graphs are locally finite.

The ends of a graph are the infinite components of a group which cannot be separated by a finite set. More precisely, an end $\xi$ is a function from finite sets to infinite connected components of their complement so that $\xi(F) \cap \xi(F') \neq \vide$ (for any $F$ and $F'$).

Given a graph $G$, a real-valued function $f$ on its vertices $V$ is said to be harmonic if it satisfies the mean value property 
\[
\forall v \in V, \; f(v) = \tfrac{1}{\deg(v)}\sum_{w \sim v} f(w). 
\]
where $v$ is the degree (or valency) of $v$. The gradient of $f$ is the function on the edges $(v,w)$ defined by $\nabla f(v,w) = f(w) - f(v)$. The square of the $\ell^2$-norm of the gradient is often referred to as the energy of the function.

The main result here is:
\begin{prop}\label{lecor}
Assume $H$ is infinite and has at most two ends, $L$ has at least one edge, $L$ has two ends or less and that both $L$ and $H$ are locally finite, then there are no non-constant harmonic functions with gradient in $\ell^p$ in $L \wr H$ for any $p\in [1,\infty[$.
\end{prop}
This result is in contrast with the fact that lamplighter graphs have bounded harmonic functions as soon as $H$ is not recurrent. Indeed, a bounded function has necessarily its gradient in $\ell^\infty$.

In fact, Proposition \ref{lecor} uses (and, when the graphs have bounded valency, is equivalent to) the vanishing of the reduced $\ell^p$ cohomology in degree one, see \cite{moi-poiss} for definitions. 
The proof of Proposition \ref{lecor} is essentially a particular case of \cite[Question 1.6]{moi-poiss}. This answers partially questions which may be found (in different guises) in 
Georgakopoulos \cite[Problem 3.1]{Agelos} and
Gromov \cite[\S{}8.$A_1$.($A_2$), p.226]{Gro}. 
Regarding \cite{Agelos}, Proposition \ref{lecor} seems hard to adapt to cases with infinitely many ends, but covers all $p$ (instead of $p=2$).

As for \cite{Gro}, the question there concerns other types of graphs; for lamplighter graphs of Cayley graphs the answer to this question is essentially complete. Indeed, a wreath product (\ie lamplighter \emph{group}) is amenable exactly when the lamp and space groups are amenable. Since amenable groups have at most $2$ ends, Proposition \ref{lecor} shows the reduced $\ell^p$-cohomology of any amenable wreath product is trivial. Note that Martin \& Valette \cite[Theorem.(iv)]{MV} show this is still true when $L$ is not amenable and has infinitely many ends (and $H$ is infinite). 

Proposition \ref{lecor} extends probably to graphs with finitely many ends. To do this one would need to answer the following question. Assume $\mathcal{G}$ is the set of graphs obtained by taking a cycle and attaching to it finitely many (half-infinite) rays. Is the lamplighter graph $L \wr H$ with $L,H \in \mathcal{G}$ Liouville? This seems to follow from classical consideration of Furstenberg (coupling), since both $H$ and $L$ are recurrent.

Our other application concerns direct product. Given two graph $H_1=(X_1,E_1)$ and $H_2=(X_2,E_2)$, the direct product $H_1 \times H_2$ is defined as follows. Its vertices are elements of $X_1 \times X_2$. Two vertices $(x_1,x_2)$ and $(x_1',x_2')$ are adjacent if either either $x_1 \sim x_1'$ or $x_2 \sim x_2'$ but not both.
\begin{prop}\label{prodi}
Assume $G$ is a direct product of graphs $H_1 \times H_2$, so that $H_1$ has $1$ or $2$ ends and $H_2$ is a Cayley graph with volume growth at least polynomial of degree $d$, then there are no non-constant harmonic functions with gradient in $\ell^p$ for all $p < \tfrac{d+1}{2}$.
\end{prop}
$H_1$ is only locally finite, but $H_2$ will be of bounded valency.
This generalises a result of Martin \& Valette \cite[Theorem.(v)]{MV} (on product of groups and which requires that one group in the direct product be non-amenable):
\begin{cor}\label{corprod}
Let $\Gm$ be a direct product of infinite [finitely generated] groups. Then there are no non-constant harmonic functions with gradient in $\ell^p$ in any Cayley graph of $\Gm$ (and the reduced $\ell^p$ cohomology in degree $1$ is trivial for all $p \in [1,\infty[$).
\end{cor}
Proposition \ref{lecor} and Corollary \ref{corprod} also have consequences on the cohomology of Hilbertian representations with $\ell^p$-coefficients, see \cite[Corollary 2.6]{GJ}. The same can be said for some representations given by $G \curvearrowright \mathsf{L}^q$ (with coefficients in $\ell^p$) modulo the following remark:
\begin{rmk}
There is a non-linear analogue of harmonic equations called $p$-harmonic equation (with $p \in ]1,\infty[$). The proofs of the Propositions \ref{lecor} and \ref{prodi} also apply to $q$-harmonic functions with gradient in $\ell^p$. Indeed, $q$ is irrelevant, since only the fact that harmonic functions satisfy the maximum principle is required to conclude (and $q$-harmonic functions also satisfy the maximum principle). 
\end{rmk}

{\it Acknowledgments:} The authors wishes to thank A.~Georgakopoulos for mentioning the existence of the work of Thomassen \cite{Tho}, thus allowing to apply the current results outside the class of groups.

\section{Preliminaries}

Let $\D^p(G)$ be the space of functions on the vertices of the graph $G$ with gradient in $\ell^p$ and $\HD^p(G)$ be the subset of $\D^p(G)$ consisting of functions which are furthermore harmonic. Lastly, $\BHD^p(G)$ are the bounded functions in $\HD^p(G)$.
The notation $\HD^p(G) \simeq \rr$  means that the only functions in $\HD^p(G)$ are constants.

For $F \subset X$ a subset of the vertices, let $\del F$ be the edges between $F$ and $F^\comp$. Let $d \in \rr_{\geq 1}$. Then, a graph $G=(X,E)$ has 
\[
\IS_d \text{ if there is a }  \kappa >0 \text{ such that for all finite } F \subset X, \; |F|^{(d-1)/d} \leq \kappa |\del F|.
\]
Quasi-homogeneous graphs with a certain (uniformly bounded below) volume growth in $n^d$ will satisfy these isoperimetric profiles, see Woess' book \cite[(4.18) Theorem]{Woe}. For example, the Cayley graph of a group $G$ satisfies $\IS_d$ \emph{for all} $d$ if and only if $G$ is not virtually nilpotent.

An important ingredient of the proofs is a result from \cite{moi-poiss}. Let $B_n$ be a sequence of balls in the graph with the same centre and $B_n^\comp$ its complement. On a connected graph, a function $f:X \to \rr$ takes only one value at infinity if $\exists c \in \rr$ so that $\forall \eps>0, \exists n_\eps$ satisfying $f(B_{n_\eps}^\comp) \subset [c-\eps,c+\eps]$. Define for $p \geq 1$:
\begin{enumerate}\renewcommand{\labelenumi}{$(\arabic{enumi}_p)$} \setlength{\itemsep}{0em}
 \item The reduced $\ell^p$-cohomology in degree one vanishes (for short, $\ssl{\ell^qH}^1 = \{0\}$).
 \item All functions in $\D^p(G)$ take only one value at infinity.
 \item There are no non-constant functions in $\HD^p(G)$.
 \item There are no non-constant functions in $\BHD^p(G)$.
\end{enumerate}
For the record, note that $(1_1) \iff (2_1) \iff$the number of ends is $>1$ (see \cite[Proposition A.2]{moi-poiss}). 
Let us sum up \cite[Theorem 1.2]{moi-poiss} here again:
\begin{teo}\label{monteo}
Assume a graph $G$ is of bounded valency and has $\IS_d$. For $1< p < d/2$, $(1_p) \iff (2_p) \implies (3_p) \implies (4_p)$ and, for $q \geq \frac{dp}{d-2p}$, $(4_q) \implies (1_p)$.

If $G$ has $\IS_d$ for all $d$, then ``$\forall p \in ]1,\infty[, (i_p)$ holds'' where $i \in \{1,2,3,4\}$ are four equivalent conditions.
\end{teo}
The important corollary of the above theorem (see \cite[Corollary 4.2.1]{moi-poiss}) is that if a graph $G$ has a spanning subgraph which is Liouville and has $\IS_d$ for some $d$ (resp. for all $d$), then $(1_p)$, $(2_p)$ and $(3_p)$ hold for any $p< d/2$ (resp. for all $p < \infty$). Indeed, Liouville implies that $(4_q)$ holds for all $q$, and the condition $(2_p)$ passes from a spanning subgraphs to the whole graph.

\section{Proof of Proposition \ref{lecor}}\label{s-proof}

The main second ingredient of the proof of Proposition \ref{lecor} is the following. Let $G_0=L \wr H$ the lamplighter graph where $L$ is either finite or a Cayley graph of $\zz$ and $H$ is a Cayley graph of $\zz$. For our current purpose it will suffice to note that $G_0$ has $\IS_d$ for any $d\geq 1$, see Woess' book \cite[(4.16) Corollary]{Woe}. 
A second important ingredient is that, using Kaimanovich \cite[Theorem 3.3]{Kaimano}, $G_0$ is Liouville, \ie a bounded harmonic function is constant.

The proof will be split in three steps for convenience.

{\bf Step 1 -} Assume that $H$ and $L$ have bounded valency. 
Note that if a spanning subgraph of $G$ has $\IS_d$, it implies that $G$ has $\IS_d$. Summing up, if a graph $G$ admits $G_0$ as a subgraph then 
$(1_q)$ holds in $G$ for any $q<\infty$ and, equivalently, $(3_p)$ holds in $G$ for any $p<\infty$.

It is also possible to work only up to quasi-isometry: if two graphs of bounded valency $\Gamma$ and $\Gamma'$ are quasi-isometric, then they have the same $\ell^p$-cohomology (in all degrees, reduced or not), see  \'Elek \cite[\S{}3]{El-qi} or Pansu \cite{Pan-qi}.

Recall that the $k$-fuzz of a graph $G$, is the graph $G^{[k]}$ with the same vertices as $G$ but now two vertices are neighbours in $G^{[k]}$ if their distance in $G$ is $\leq k$. $G^{[2]}$ is often called the square of $G$.

Lastly, using either Thomassen \cite{Tho} or Seward \cite[Theorem 1.6]{Sew}, the graphs $L$ and $H$ in Proposition \ref{lecor} are bi-Lipschitz equivalent to graphs containing a spanning line (or cycle if the graph is finite). In fact, this bi-Lipschitz equivalence is given by taking the $k$-fuzz of these graphs. An interested reader could probably show that $k=4$ is sufficient. This means that $L \wr H$ is bi-Lipschitz equivalent (and so quasi-isometric) to a graph containing $G_0$. This finishes the proof of Proposition \ref{lecor} when $H$ and $L$ both have bounded valency.

{\bf Step 2 -} Assume from now on that both $H$ and $L$ have connected spanning subgraphs of bounded valency, say $H'$ and $L'$ respectively. If there is a non-constant $f \in \HD^p(G)$ (where $G = L \wr H$). Then $f$ is not constant at infinity.
Indeed, since $f$ is harmonic, the maximum principle would then imply that $f$ is constant. 

But $f$ is also a function on the vertices of $G' = L' \wr H'$ and it is also in $\D^p(G')$ (because deleting edges only reduces the $\ell^p$ norm of the gradient). So $(2_p)$ cannot hold on $G'$. On the other hand $G'$ contains $G_0$ up to quasi-isometry and hence $\ssl{\ell^pH}^1(G') = \{0\}$. However, by Theorem \ref{monteo} above,
``$(1_p)$ for all $p$'' implies ``$(2_p)$ for all $p$''.

{\bf Step 3 -} Now assume $H$ and $L$ are only locally finite. The result of Thomassen \cite{Tho} still implies that (for some $k$) the $k$-fuzz of $H$ and $L$ have a spanning line (or cycle if the graph $L$ is finite). However, given a function $f \in \D^p(G)$, it may no longer be in $\D^p(G^{[k]})$ if $k>1$ and $G$ does not have bounded valency. To circumvent this problem, construct a graph $H'$ by adding (when necessary) the edges of the spanning line in $H^{[k]}$. Construct $L'$ similarly.

Given $f \in \D^p(G)$ where $G = L \wr H$, one has that $f \in \D^p(G')$ with $G' = L' \wr H'$. Indeed, in passing from $G$ to $G'$ at most four edges are added to each vertex and the gradient along these edge is expressed as a sum of $k$ values of the gradient of $f$ on $G$. The triangle inequality ensures that the $\ell^p$-norm of $\nabla f$ (on $G'$) is at most $(4k+1)$ times the $\ell^p$-norm of the gradient of $f$ on $G$.

This last reduction yields the conclusion. Indeed, if there is an $f \in \HD^p(G)$ which is not constant, then there is an $f \in \D^p(G')$ which takes different values at infinity. This is however excluded by step 2.

\section{Proof of Proposition \ref{prodi} and Corollary \ref{corprod}}\label{s-prod}

The main second ingredient for the proof of Proposition \ref{prodi} is that if $G$ is a Cayley graph of a [finitely generated] group and this group has infinitely many finite conjugacy class (\eg infinite center) then $\ssl{\ell^pH}^1(G) = \{0\}$ (there are many possible proofs: see Kappos \cite[Theorem 6.4]{Kap}, Martin \& Valette \cite[Theorem 4.3]{MV} Puls \cite[Theorem 5.3]{Puls-Can}, Tessera \cite[Proposition 3]{Tes} or \cite[Theorem 3.2]{moi-trans}).

\begin{proof}[Proof of Proposition \ref{prodi}]
Let $\Gm$ be the group whose Cayley graph is $H_2$, let $\Gm_0 = \zz \times \Gm$ and let $G_0$ be the direct product of the bi-infinite line and $H_2$ (a Cayley graph of $\Gm_0$). By the result quoted in the previous paragraph, $\ssl{\ell^pH}^1(G_0) = \{0\}$. The growth condition (see Woess' book \cite[(4.16) Corollary]{Woe}), implies that $G_0$ has $\IS_{d+1}$. By Theorem \ref{monteo}, one deduces that $G$ has no non-constant harmonic functions with gradient in $\ell^p$ for $p < \frac{d+1}{2}$.

To realise $G_0$ as a spanning subgraph, the arguments are absolutely identical to those of the proof of Proposition \ref{lecor} (\S{}\ref{s-proof} above).
\end{proof}

\begin{proof}[Proof of Corollary \ref{corprod}]
The proof 
 requires to distinguish two cases:
\begin{itemize}\setlength{\itemsep}{0em}  \renewcommand{\labelitemi}{ $\cdot$ }
 \item if one of the two groups (say $\Gm_2$) is not virtually nilpotent, then its Cayley graphs have $\IS_d$ for all $d$. By Theorem \ref{monteo}, ``$(3_p)$ for all $p$'' is equivalent to ``$(1_p)$ for all $p$'' (which does not depend on the generating set). Take a generating set so the graph is a direct product and take $H_2$ to be a Cayley graph of $\Gm_2$. Apply Proposition \ref{prodi} to conclude.
 \item if both groups are virtually nilpotent, so is the direct product. Then it is well-known that there are no non-constant harmonic functions with gradient in $c_0$ (see for example \cite[Lemma 5]{moi-lpharm}) and even no non-constant functions with sublinear growth (see Hebisch \& Saloff-Coste \cite[Theorem 6.1]{HSC}).
 Note that in this second case, one still has that, $(1_p)$ holds $\forall p \in ]1,\infty[$ . \qedhere
\end{itemize}
\end{proof}

\end{document}